\documentclass{article}
\usepackage{amsmath, amssymb, graphics, url, multirow, subfig, amsthm}
\begin{document}
\newtheorem{theorem}{Theorem}[section]
\newtheorem{conjecture}{Conjecture}[section]
\newtheorem{proposition}{Proposition}[section]
\newtheorem{definition}{Definition}[section]
\title{An Unconditional Improvement to the Running Time of the Quadratic Frobenius Test}
\author{Jon Grantham}
\date{\today}
\maketitle
\begin{abstract}
In a 2006 paper, Damg{\aa}rd and Frandsen designed a faster version of the Quadratic Frobenius Test. This test assumes the Extended Riemann Hypothesis
in order to find small nonresidues, which allow construction of quadratic extensions with faster arithmetic. In this paper, I describe a version of the test 
using small nonresidues, without assuming any unproven hypothesis.
\end{abstract}

\section{Introduction}
From Fermat's little theorem, we know that for any odd prime $p$ and $a$ with $p\not{|}a$, $a^{p-1}\equiv 1\pmod p$. By its contrapositive, we know that any $n$ with $a^{n-1}\not\equiv 1\pmod n$ is composite; this is known as the {\bf Fermat pseudoprime test}.

Let $p=2^rs+1$ Because there are only two square roots of $1$ modulo $p$, we know that $a^s\equiv 1$ or $a^{2^js}\equiv -1$ for some $0\le j \le r-1$. The contrapositive of this fact is the basis of the {\bf strong pseudoprime test}.

This test has long been used as a fast way to prove the compositeness of integers. The test itself takes $(1+\operatorname{o}(1))\log_2n$ multiplications modulo $n$. Monier \cite{monier} and Rabin \cite{rabin} quantified the worst-case error bound in 1980 by showing that a composite will pass the test to at most $1/4$ of the bases $a$. (This bound is sharp.)

Pseudoprime tests in terms of recurrence sequences exist. Let $U_n(P,Q)$ be a Lucas sequence, where $P$ and $Q$ are integers, $U_0=0$, $U_1=1$ and $U_n=PU_{n-1}-QU_{n-2}$. Let $\Delta=P^2-4Q$.  For a prime $p\nmid 2Q\Delta$, we have $\smash[b]{U_{p-\left(\frac\Delta{p}\right)}}\equiv 0 \bmod p$. The contrapositive of this fact gives the {\bf Lucas pseudoprimality test}.

Also in 1980, Baillie and Wagstaff \cite{bw}, along with Pomerance, Selfridge and Wagstaff \cite{psw}, recognized the utility of combining Lucas pseudoprime tests with Fermat or strong pseudoprime tests. In particular,
the combination of a Fermat pseudoprime test with a Lucas pseudoprime test is often called a Baillie-PSW test.

In order to get a worst-case error bound in the style of Monier and Rabin \cite{rabin}, in 1998 I \cite{meone} introduced a randomized version of the Baillie-PSW test, known as the Quadratic Frobenius Test (QFT). The QFT takes
time comparable to three iterations of the strong pseudoprime test, but produces a better worst-case error bound. The QFT is expressed in terms of quadratic extensions of $\mathbb{Z}/n\mathbb{Z}$ rather than second-order recurrence sequences. Theorem 5.6 of \cite{metwo} shows that composites that pass the QFT also pass a Lucas pseudoprime test. Although unfortunately not explicit in that paper, it is possible to show that composites passing the QFT also pass a Fermat test.

Subsequent improvements to the QFT were made by Zhang \cite{zhang} and M{\"u}ller \cite{muller1}, \cite{muller3}.

In their 2006 paper \cite{df}, Damg{\aa}rd and Frandsen made two improvements to the Quadratic Frobenius Test (QFT). They improved the worst-case and average-case error bounds by looking at $24$th roots of unity. They also
improved the running time of the algorithm, under the assumption of the Extended Riemann Hypothesis (ERH), by constructing a quadratic extension with small coefficients.

In this paper, I examine the second improvement, but drop the ERH assumption, at the cost of a somewhat larger quadratic extension.

Additionally, the original QFT analysis assumed that a modular multiplication and a modular squaring took equivalent time. Other authors have not followed this convention; I explore this distinction.

\section{The Quadratic Frobenius Test and Its Reformulation}

The original QFT is as follows.

\begin{definition}
Suppose $n>1$ is odd, $\left(\frac{b^2+4c}n \right)=-1$, and
$\left(\frac{-c}n \right)=1$. Let $B=50000$.
The {\bf Quadratic Frobenius Test (QFT) with parameters $(b,c)$} consists
of the following.

\begin{enumerate}
\item Test $n$ for divisibility by primes less than or equal to
$\min\{B,\sqrt{n}\}$.
If it is divisible by one of these primes, declare $n$ to be
composite and stop.
 
\item Test whether $\sqrt{n}\in\mathbb{Z}$.  If it is, declare $n$ to be composite
and stop.

\item Compute $x^{\frac{n+1}2}$ mod $(n,x^2-bx-c)$.  If
$x^{\frac{n+1}2}\not\in\mathbb{Z}/n\mathbb{Z}$, declare $n$ to be composite and
stop.
 
\item Compute $\smash[t]{x^{n+1}}$ mod $(n,x^2-bx-c)$.  If 
$\smash[t]{x^{n+1}}\not\equiv -c$,
declare $n$ to be composite and stop.

\item Let $\smash[t]{n^2-1=2^rs}$, where $s$ is odd.
If $x^s\not\equiv 1$ mod $(n,x^2-bx-c)$,
and $\smash[t]{x^{2^js}\not\equiv -1}$ for all $0\le j\le
r-2$, declare $n$ to be composite and stop.

\item If $n$ is not declared composite in Steps 1--5, declare $n$ to be a
probable prime.

\end{enumerate}
\end{definition}

This test is based on finding a quadratic extension whose discriminant has Jacobi symbol $-1$, and then 
testing whether the $n$th power map on $x$ behaves like the Frobenius map. Theorem 3.4 of \cite{meone}
shows that the running time is at most three times that of an ordinary (Fermat) probable prime test.

Instead, one could choose an
extension $x^2-c$, and then choose an element of that extension. This is the approach introduced by
Damg{\aa}rd and Frandsen.

\begin{definition}
Suppose $n>1$ is odd, $\left(\frac{b^2-ca^2}n\right)=1$, and
$\left(\frac cn\right)=-1$. Let $z=ax+b$. Let $B=50000$.
The {\bf reformulated Quadratic Frobenius Test (rQFT) with parameters $(a,b,c)$} consists
of the following.

\begin{enumerate}
\item Test $n$ for divisibility by primes less than or equal to
$\min\{B,\sqrt{n}\}$.
If it is divisible by one of these primes, declare $n$ to be
composite and stop.
 
\item Test whether $\sqrt{n}\in\mathbb{Z}$.  If it is, declare $n$ to be composite
and stop.

  \item Compute $z^{\frac{n+1}2}$ mod $(n,x^2-c)$. If
$z^{\frac{n+1}2}\not\in\mathbb{Z}/n\mathbb{Z}$, declare $n$ to be composite and
stop.
  \item Compute $\smash[t]{z^{n+1}}$ mod $(n,x^2-c)$.  If 
$\smash[t]{z^{n+1}}\not\equiv b^2-ca^2$,
declare $n$ to be composite and stop.

\item Let $\smash[t]{n^2-1=2^rs}$, where $s$ is odd.
If $z^s\not\equiv 1$ mod $(n,x^2-c)$,
and $\smash[t]z^{2^js}\not\equiv -1$ mod $(n,x^2-c)$ for all $0\le j\le
r-2$, declare $n$ to be composite and stop.

\item If $n$ is not declared composite in Steps 1--5, declare $n$ to be a
probable prime.

\end{enumerate}
  
\end{definition}

It is easy to pass between the two formulations with a change of variables.

An advantage of the second formulation is that if one can find small $c$ with $\left(\frac cn\right)=-1$, then the arithmetic is faster. (Alternatively, if one finds a nontrivial $c$ with  $\left(\frac cn\right)=0$, then $n$ is proven composite via factorization.)
  
\section{Unconditionally Finding Small Quadratic Extensions}

Damg{\aa}rd and Frandsen \cite{df} assume the Extended Riemann Hypothesis to find small $c$ with $\left(\frac cn\right)=-1$.

We can, however, unconditionally find such a $c$ small enough to give us a computational advantage. 

\begin{theorem}
  If $n$ is a sufficiently large composite number that is not a square, and $\delta>\frac 1{3\sqrt{e}}$, a positive proportion of the numbers $0<c<n^{\delta}$ have $\left(\frac{c}n\right)\not=1$.
\end{theorem}

\begin{proof}
First, a ``Burgess bound'' shows that the sum of Jacobi symbols modulo $n$ is small when the sum is taken to a power of $n$ near $n^{1/3}$. Theorem A of \cite{burgess} shows that for any $\epsilon>0$, $\sum_{k<n^\gamma} \left(\frac kn\right) < (n^\gamma)^{2/3}n^{1/9+\epsilon}$. We can then take $\gamma=1/3+6\epsilon$ to get $\sum_{k<n^\gamma} \left(\frac kn\right)~<~n^{1/3+{5}\epsilon}=\operatorname{o}(n^\gamma)$.

Then a result of Granville and Soundararajan allows the power of $n$ to be reduced to that in the statement of the theorem. The result was listed as Theorem 4.1 of \cite{banks} as an ``unpublished result,'' but the arguments are
given in \cite{gransound}, particuarly Corollary 1.8 The formulation in \cite{banks}, however, is the one needed here. That states that if $\sum_{m\le x} f(n)=\operatorname{o}(x)$ and $\alpha>1/\sqrt{e}$, then $\left|\sum_{m\le x^{\alpha}} f(m) \right| < Cx^{\alpha}$, for some $C<1$.
For any $\delta>\frac 1{3\sqrt{e}}$ we can choose $\gamma$ and $\alpha$ such that $\delta=\alpha\gamma$ and get the bound in the theorem.

\end{proof}

\section{The Cost of QFTs}

Atkin \cite{atkin} defined a ``Selfridge unit'' (SU) as the time required to perform $(1+\operatorname{o}(1))\log_2 N$ modular squarings on a number of size $N$. I \cite{meone} adapted that to the ``selfridge'', the time required to perform $\log_2 N$ modular
multiplications (whether they were squarings or not). Atkin was displeased by this simplification; in fact, he made the assumption that one modular multiplication was equal to the cost of two modular squarings (MSQs).

Damg{\aa}rd and Frandsen \cite{df}, however, assume that each modular multiplication costs $1.3$ MSQs. The discrepancy between the ratio they use and the one Atkin used could be used in support of my assumption that the different costs of a squaring and a multiplication are implementation-dependent, and should be ignored. On the other hand, neither ratio is $1$, which argues in favor of tracking the difference between squarings and multiplications. I do so below.

The original QFT takes two modular squarings and one modular multiply for each of $(1+\operatorname{o}(1))\log_2 N$ operations in the quadratic extension. If we assume that each multiply is $m$ MSQs, that is a cost of $2+m$ MSQs per operation. The reformulation of Damg{\aa}rd and Frandsen takes 3 modular multiplies ($3m$ MSQs), or 2 modular multiplies ($2m$) assuming the ERH. Using Theorem 3.1 allows one of the modular multiplies to be with a number of size $\delta N$, which cuts the time required to $\delta m$ MSQs, for a total of $(2+\delta)m$ MSQs.

Under Atkin's original weighting ($m=2$), the QFT costs $4$ SUs, $6$ SUs under the reformulation and $4$ under the ERH. Neither that cost nor $(2+\delta)2\approx 4.4$ SUs from Theorem 3.1 is an improvement.

Under the Damg{\aa}rd-Frandsen approach, where $m=1.3$, the original QFT is $3.3$ SUs, the reformulation is $3.9$ SUs. The ERH brings that down to $2.6$ SUs, while Theorem 3.1 allows $\approx 2.86$ SUs.

Using the definition of selfridges from \cite{meone}, both the original QFT and the reformulation cost $3$ selfridges, while the ERH brings that down to $2$ selfridges, and the result from the previous section allows $2+\delta\approx 2.2$ selfridges.

\section*{Acknowledgments}

Thanks to Paul Pollack for valuable e-mail exchanges about Burgess bounds. Thanks to Xander Faber for helpful comments on an earlier draft of this paper. 

\bibliography{pseudo4}
\bibliographystyle{monthly}

\end{document}